\documentclass{amsart}

\usepackage{amssymb,amsmath,amsthm,latexsym,amscd,amsfonts,mathrsfs,mathtools}
\usepackage{epsfig}
\usepackage{graphicx}
\usepackage{bbold}
\usepackage{color}
\usepackage[all,cmtip]{xy}

\newtheorem{Theorem}{\bf Theorem}
\newtheorem{lemma}[Theorem]{\bf Lemma}
\newtheorem{proposition}[Theorem]{\bf Proposition}
\newtheorem{corollary}[Theorem]{\bf Corollary}

\newtheorem{definition}[Theorem]{\bf Definition}
\newtheorem{example}[Theorem]{\bf Example}

\newtheorem{theorem}[Theorem]{\bf Theorem}

\def\qed{\hfill$\Box$}
\def\scfig #1 #2 {\resizebox{#2}{!}{\includegraphics{#1}}}

\newcommand{\be}{\begin{equation}}
\newcommand{\ee}{\end{equation}}
\newcommand{\fkm}{\mathfrak m}

\numberwithin{equation}{section}

\begin{document}

\title[Indecomposable integrally closed modules]{Indecomposable integrally closed modules of rank 3 over two-dimensional regular local rings}

\author{Futoshi Hayasaka}
\address{Department of Environmental and Mathematical Sciences, Okayama University, 3-1-1
Tsushimanaka, Kita-ku, Okayama, 700-8530, JAPAN}
\author{Vijay Kodiyalam}
\address{The Institute of Mathematical Sciences, Chennai, India and Homi Bhabha National Institute, Mumbai, India}
\email{hayasaka@okayama-u.ac.jp,vijay@imsc.res.in}

\begin{abstract}
We characterise ideals in two-dimensional regular local rings that arise as ideals of maximal minors of indecomposable integrally closed modules of rank three.
\end{abstract}


\maketitle

\section{Introduction}

The theory of integrally closed ideals in two-dimensional regular local rings was developed by Zariski in \cite{Zrs1938, ZrsSml1960}. The fundamental structural results of this theory are the 
product and unique factorisation theorems which assert that the product of any two integrally closed ideals is again integrally closed, and that any non-zero integrally closed ideal can be expressed uniquely (except for ordering) as a product of simple integrally closed ideals. 
Since then, this theory has been studied and extended by several authors in different directions. In \cite{Kdy1995}, 
the second named author initiated the study of integrally closed modules 
over two-dimensional regular local rings, 
and generalised the classic theory in this direction. 
The non-triviality of this generalisation relies on the 
existence of indecomposable integrally closed modules of rank at least two. 
In the paper \cite{Kdy1995}, it was proved that 
there exist indecomposable integrally closed modules of arbitrary rank 
by showing that one such module can be constructed from any given simple integrally closed ideal. 
In \cite{Hys2020, Hys2022}, the first named author constructed 
indecomposable integrally closed modules of arbitrary rank from given ideals in a large class of integrally closed monomial ideals which are not necessarily simple.
Thus we have a large class of indecomposable integrally closed modules whose ideals of maximal minors 
are simple integrally closed ideals or non-simple integrally closed monomial ideals.

In this paper we characterise ideals of two-dimensional regular local rings that arise as the ideal of maximal minors of a matrix whose columns generate an indecomposable integrally closed module of rank three. This is a continuation of the work done in \cite{HysKdy2022} which answered the same question in the rank two case, which itself was motivated by the examples of monomial ideals in \cite{Hys2020} and \cite{Hys2022}. While it is natural to ask for a similar characterisation for all ranks, the cases of rank two and three are somewhat special since any decomposable module of these ranks has an ideal as direct summand. We hope to deal with the higher rank case in a future publication.

We will primarily do two things. First, in Section \ref{exissec}, 
we establish a general ``existence result" for arbitrary rank $r$. This result is of the type that asserts that ideals satisfying certain conditions occur as the ideal of maximal minors of indecomposable integrally closed modules of rank $r$. The result depends on the construction of a specific module which is introduced in Section  \ref{prelimsec}.
Secondly, we will prove some ``non-existence results" in rank three in Section \ref{nonexissec}. 
These assert that certain ideals cannot be the ideal of maximal minors of an indecomposable integrally closed module of rank three. Putting these together gives us the desired characterisation - Theorem \ref{nonexistence} - of ideals that arise as the ideal of maximal minors of an indecomposable integrally closed module of rank three. 
On the way to proving these results, we identify a class of integrally closed modules that we call extremal modules and study some of their properties in Section \ref{exissec}. We believe that this class of modules merits further study.

\section{Preliminary results}\label{prelimsec}

Throughout this note $(R,{\mathfrak m})$ will be a two-dimensional regular local ring with infinite residue field and $M$ will be a finitely generated torsion-free $R$-module. We will assume familiarity with the basic results on finitely generated integrally closed $R$-modules from \cite{Kdy1995} where a related question on the ideals of minors of indecomposable modules was first raised. We will denote $M^{**}$ by $F$ - which is a free $R$-module - and always regard $M$ as a submodule of $F$ generated by the columns of a suitable matrix. The ideal of maximal minors of this matrix will be denoted by $I(M)$ and is an ${\mathfrak m}$-primary ideal of $R$ (or the whole of $R$ if $M$ is itself free). By $ord(M),rk(M)$ and $\mu(M)$ we mean the order of the ideal $I(M)$, the rank of $M$ and the minimal number of generators of $M$ respectively.

Before proceeding, we recall the determinantal criterion of integral dependence which is one of the mainstays of this paper. With $M \subseteq F$ as before, an element $v \in F$ (thought of as a column vector) is in the integral closure of the module $M$ iff the determinant of any matrix with one column $v$ and the rest of the columns in $M$ is integral over $I(M)$. This is a special case of a result of Rees in \cite{Res1987}.

When $M$ has no free direct summands,
it is contained in ${\mathfrak m}F$. 
Indeed, if $M \nsubseteq \mathfrak m F$, then $M$ contains a minimal generator  $v$ of $F$, which implies that $Rv$ is a direct summand of $M$. 
So, if $M$ is indecomposable of $rk(M)=r \geq 2$, 
then $I(M) \subseteq {\mathfrak m}^r$ and so $ord(I(M)) \geq r$.
A direct corollary of the determinantal criterion is the following generalisation of Lemma 1 of \cite{HysKdy2022}, which shows that converse is certainly false.

\begin{lemma}\label{mtother} 
Let $r \geq 2$. Then
the ideal ${\mathfrak m}^r$ is not $I(M)$ for any indecomposable integrally closed $R$-module $M$ of rank $r$.
\end{lemma}

\begin{proof}
Suppose $M$ is integrally closed of rank $r$ without free direct summands and that $I(M) = {\mathfrak m}^r$. With $F = M^{**}$, as usual, $F \cong R^r$ and $M \subseteq {\mathfrak m}F$. Since all entries of a matrix for $M$ lie in ${\mathfrak m}$, the determinantal criterion shows that for any $x \in{\mathfrak m}$, the vectors
$xe_i$ are integral over $M$ and hence are in $M$. Here $e_i$ denotes 
the vector whose $i$th entry is $1$ and the others are $0$. 
Thus $M \supseteq {\mathfrak m}F$ as well and so 
$M = {\mathfrak m}F = {\mathfrak m}^{\oplus r}$ and therefore decomposable.
\end{proof}

The existence results depend on the construction of a specific module of rank $r$ associated with an ${\mathfrak m}$-primary ideal $I$ with $ord(I) = m \geq r$ and $\mu(I) = n+1$.
Before defining this module, we recall the fact that for any $\mathfrak m$-primary ideal $I$, the inequality $\mu(I) \leq ord(I)+1$ always holds, and 
the equality $\mu(I)=ord(I)+1$ holds if $I$ is integrally closed (or even, contracted) - see Chapter 14 of \cite{HnkSwn2006}. 
Thus $n \leq m$.

Write $I = (a_1,a_2,\cdots,a_{n+1})$ for a minimal generating set and for
each $j \in \{1,\cdots,n+1\}$, let $a_j = a_{1,j}x^{r-1} + a_{2,j}x^{r-2}y + \cdots + a_{r,j}y^{r-1}$ where ${\mathfrak m}=(x,y)$. All $a_{i,j} \in {\mathfrak m}$ necessarily since $I \subseteq {\mathfrak m}^r$.
Let $M_r(I)$ denote the module generated by the columns of the $r \times (n+r)$ matrix
\begin{equation}\label{spmod}
\left[
\begin{array}{ccccrrrr}
a_{1,1} & a_{1,2}  & \cdots & a_{1,n+1}  & y & 0 &   0 & 0\\
a_{2,1} & a_{2,2}  & \cdots & a_{2,n+1}  & -x & y &   0 & 0\\
\vdots & \vdots &  & \vdots  & \vdots & \ddots  & \ddots & \vdots\\
a_{r-1,1} & a_{r-1,2} & \cdots & a_{r-1,n+1}  & 0 & 0   & -x & y\\
a_{r,1} & a_{r,2} & \cdots & a_{r,n+1} & 0 & 0   & 0 & -x
\end{array}
\right].
\tag{$\ast$}
\end{equation}
Note that the notation $M_r(I)$ is justified. For first, it does not depend on the specific way of representing the generators of $I$ in terms of the $x^{r-1},\cdots,y^{r-1}$ since any two such representations of a generator differ by a linear combination of the last $r-1$ columns as these span all relations among the $x^{r-1},\cdots,y^{r-1}$. Secondly, it does not depend on the minimal generating for $I$ chosen, since changing to a different one is equivalent to performing column operations on the first $n+1$ columns. Thirdly, up to isomorphism, it does not depend on the specific chosen generating set $x,y$ of ${\mathfrak m}$ as follows easily from Lemma \ref{indep} below. The notation $Sym^k(Q)$ for a $2 \times 2$ matrix $Q$ stands for the following $(k+1)\times(k+1)$ matrix. Consider the linear forms $X,Y$ in two variables $U,V$ determined by the columns of  $Q$ as $X=aU+bV$, $Y=cU+dV$ where
$$
Q =
\left[
\begin{array}{cc}
a & c\\
b & d
\end{array}
\right].
$$
The columns of $Sym^{k}(Q)$ are then the coefficients of $U^{k},\cdots,V^{k}$ in $X^{k},\cdots,Y^{k}$. 
Equivalently, $Sym^{k}(Q)$ is the unique matrix that satisfies
$$
\left[
\begin{array}{cccc}
X^{k} & X^{k-1}Y & \cdots & Y^k
\end{array}
\right]
=
\left[
\begin{array}{cccc}
U^{k} & U^{k-1}V & \cdots & V^k
\end{array}
\right]
Sym^k(Q)
$$

\begin{lemma}\label{indep} Suppose that ${\mathfrak m} = (u,v)$ where,
$$
\left[
\begin{array}{cc}
x & y
\end{array}
\right]
=
\left[
\begin{array}{cc}
u & v
\end{array}
\right] Q
=
\left[
\begin{array}{cc}
u & v
\end{array}
\right]
\left[
\begin{array}{cc}
a & c\\
b & d
\end{array}
\right].
$$
for a matrix $Q \in GL_2(R)$.
Fix a minimal generating set $(a_1,a_2,\cdots,a_{n+1})$ of $I$ and consider the matrix 
$($\ref{spmod}$)$ above whose first $n+1$ columns express the $a_i$ in terms of $x^{r-1},\cdots,y^{r-1}$ and  whose last $r-1$ columns are the relations between $x^{r-1},\cdots,y^{r-1}$. Then, 
\begin{enumerate}
\item The columns of
$$
Sym^{r-1}(Q)
\left[
\begin{array}{ccccrrrr}
a_{1,1} & a_{1,2}  & \cdots & a_{1,n+1}\\
a_{2,1} & a_{2,2}  & \cdots & a_{2,n+1}\\
\vdots & \vdots &  & \vdots \\
a_{r-1,1} & a_{r-1,2} & \cdots & a_{r-1,n+1}\\
a_{r,1} & a_{r,2} & \cdots & a_{r,n+1}
\end{array}
\right]
$$ 
express the $a_i$ in terms of $u^{r-1},\cdots,v^{r-1}$, and
\item 
$$
Sym^{r-1}(Q)
\left[
\begin{array}{ccccrrrr}
 y & 0 &   0 & 0\\
 -x & y &   0 & 0\\
  \vdots & \ddots  & \ddots & \vdots\\
0 & 0   & -x & y\\
 0 & 0   & 0 & -x
\end{array}
\right]
=
det(Q)
\left[
\begin{array}{ccccrrrr}
 v & 0 &   0 & 0\\
 -u & v &   0 & 0\\
  \vdots & \ddots  & \ddots & \vdots\\
0 & 0   & -u & v\\
 0 & 0   & 0 & -u
\end{array}
\right] 
Sym^{r-2}(Q).
$$
\end{enumerate}
\end{lemma}

\begin{proof} (1) follows easily from definitions. 
To show (2) write,
\begin{eqnarray*}
\left[
\begin{array}{rrrrrrrr}
 y & 0 &   0 & 0\\
 -x & y &   0 & 0\\
  \vdots & \ddots  & \ddots & \vdots\\
0 & 0   & -x & y\\
 0 & 0   & 0 & -x
\end{array}
\right]
&=&
u
\left[
\begin{array}{rrrrrrrr}
 c & 0 &   0 & 0\\
 -a & c &   0 & 0\\
  \vdots & \ddots  & \ddots & \vdots\\
0 & 0   & -a & c\\
 0 & 0   & 0 & -a
\end{array}
\right]
+
v
\left[
\begin{array}{rrrrrrrr}
 d & 0 &   0 & 0\\
 -b & d &   0 & 0\\
  \vdots & \ddots  & \ddots & \vdots\\
0 & 0   & -b & d\\
 0 & 0   & 0 & -b
\end{array}
\right],\\
\left[
\begin{array}{rrrrrrrr}
 v & 0 &   0 & 0\\
 -u & v &   0 & 0\\
  \vdots & \ddots  & \ddots & \vdots\\
0 & 0   & -u & v\\
 0 & 0   & 0 & -u
\end{array}
\right]
&=&
u
\left[
\begin{array}{rrrrrrrr}
 0 & 0 &   0 & 0\\
 -1 & 0 &   0 & 0\\
  \vdots & \ddots  & \ddots & \vdots\\
0 & 0   & -1 & 0\\
 0 & 0   & 0 & -1
\end{array}
\right]
+
v
\left[
\begin{array}{rrrrrrrr}
 1 & 0 &   0 & 0\\
 0 & 1 &   0 & 0\\
  \vdots & \ddots  & \ddots & \vdots\\
0 & 0   & 0 & 1\\
 0 & 0   & 0 & 0
\end{array}
\right],
\end{eqnarray*}
and comparing coefficients of $u,v$ on either side, reduce to showing
\begin{eqnarray*}
Sym^{r-1}(Q)
\left[
\begin{array}{rrrrrrrr}
 c & 0 &   0 & 0\\
 -a & c &   0 & 0\\
  \vdots & \ddots  & \ddots & \vdots\\
0 & 0   & -a & c\\
 0 & 0   & 0 & -a
\end{array}
\right]
&=&
det(Q)
\left[
\begin{array}{rrrrrrrr}
 0 & 0 &   0 & 0\\
 -1 & 0 &   0 & 0\\
  \vdots & \ddots  & \ddots & \vdots\\
0 & 0   & -1 & 0\\
 0 & 0   & 0 & -1
\end{array}
\right]
Sym^{r-2}(Q),\\
Sym^{r-1}(Q)
\left[
\begin{array}{rrrrrrrr}
 d & 0 &   0 & 0\\
 -b & d &   0 & 0\\
  \vdots & \ddots  & \ddots & \vdots\\
0 & 0   & -b & d\\
 0 & 0   & 0 & -b
\end{array}
\right]
&=&
det(Q)
\left[
\begin{array}{rrrrrrrr}
 1 & 0 &   0 & 0\\
 0 & 1 &   0 & 0\\
  \vdots & \ddots  & \ddots & \vdots\\
0 & 0   & 0 & 1\\
 0 & 0   & 0 & 0
\end{array}
\right]
Sym^{r-2}(Q).
\end{eqnarray*}

The equations are verified by premultiplying each with $\left[
\begin{array}{cccc}
U^{r-1} & U^{r-2}V & \cdots & V^{r-1}
\end{array}
\right]$ and using the defining property of $Sym^k(Q)$ (for $k=r-1$ on the left hand side and for $k=r-2$ on the right hand side).
\end{proof}

To summarise, we have proved the following proposition.

\begin{proposition}
Let $I$ be an $\mathfrak m$-primary ideal 
with $ord(I) = m$, $\mu(I)=n+1$ and let $r \leq m$.
Then, up to isomorphism, the module $M_r(I)$ defined by the matrix $($\ref{spmod}$)$ above does not depend on the chosen generating system of $I$, the 
generating system of $\mathfrak m$ and the coefficients expressing each generator
in terms of $\{x^{r-1}, \dots , y^{r-1}\}$. \qed
\end{proposition}

\color{black}

Throughout this paper we will use several times, usually without specific mention, that pre- or post-multiplying a matrix by invertible matrices does not change the isomorphism class of the module generated by the columns of that matrix.
We will next prove several properties of $M_r(I)$. The notation $I_k(M)$ refers to the ideal generated by the  $k \times k$ minors of a matrix whose columns generate $M$. Thus, $I_r(M) = I(M)$ for a module $M$ of rank $r$.

\begin{proposition}\label{fittingideals} Let $M = M_r(I)$ for an integrally closed ${\mathfrak m}$-primary ideal $I$ with $ord(I) = n \geq r$.
Then, $I_k(M) = {\mathfrak m}^k$ for $k=1,2,\cdots,r-1$ and $I(M) = I_r(M) = I$.
\end{proposition}

\begin{proof} Since all $a_{i,j} \in {\mathfrak m}$, it is clear that for $k=1,2,\cdots,r-1$, $I_k(M) = {\mathfrak m}^k$ just from the minors of the last $r-1$ columns. It is also clear that $I_r(M)$ contains $I$ because for any  $j \in \{1,\cdots,n+1\}$, the $r \times r$ minor formed by taking the
$j^{th}$ column and the last $r-1$ columns is $(-1)^{r-1}a_j$.  It remains to see that $I_r(M) \subseteq I$.

For this, observe first, that by construction, the matrix equation
$$
\left[
\begin{array}{cccc}
x^{r-1} & x^{r-2}y & \cdots & y^{r-1}
\end{array}
\right]
\left[
\begin{array}{ccccrrrr}
a_{1,1} & a_{1,2}  & \cdots & a_{1,n+1}  & y & 0 &   0 & 0\\
a_{2,1} & a_{2,2}  & \cdots & a_{2,n+1}  & -x & y &   0 & 0\\
\vdots & \vdots &  & \vdots  & \vdots & \ddots  & \ddots & \vdots\\
a_{r-1,1} & a_{r-1,2} & \cdots & a_{r-1,n+1}  & 0 & 0   & -x & y\\
a_{r,1} & a_{r,2} & \cdots & a_{r,n+1} & 0 & 0   & 0 & -x
\end{array}
\right]
$$
$$
 = 
 $$
 $$
\left[
\begin{array}{cccccccc}
a_1 & a_2 & \cdots & a_{n+1} & 0 & 0 & \cdots & 0
\end{array}
\right],
$$
holds. Let $A$ be any $r \times r$ submatrix obtained by choosing some $r$ columns of generators. The previous equation implies that
$$
\left[
\begin{array}{cccc}
x^{r-1} & x^{r-2}y & \cdots & y^{r-1}
\end{array}
\right] A \subseteq 
\left[
\begin{array}{cccc}
I & I & \cdots & I 
\end{array}
\right],
$$
in the sense that each entry of the row-matrix on the left is in $I$. Multiplying by $adj(A)$ on the right gives
\begin{eqnarray*}
\left[
\begin{array}{cccc}
x^{r-1} & x^{r-2}y & \cdots & y^{r-1}
\end{array}
\right] A .adj(A) &\subseteq& 
\left[
\begin{array}{cccc}
I & I & \cdots & I 
\end{array}
\right]adj(A)\\
&\subseteq&
\left[
\begin{array}{cccc}
{\mathfrak m}^{r-1}I & {\mathfrak m}^{r-1}I & \cdots & {\mathfrak m}^{r-1}I 
\end{array}
\right],
\end{eqnarray*}
where the second inclusion follows since every entry of $adj(A)$ is in ${\mathfrak m}^{r-1}$. Hence ${\mathfrak m}^{r-1}det(A) \in {\mathfrak m}^{r-1}I$.  So $det(A) \in ({\mathfrak m}^{r-1}I:{\mathfrak m}^{r-1}) = I$ since $I$ is integrally closed.

Since $I_r(M)$ is generated by all such $det(A)$, we see that $I_r(M) \subseteq I$ too.
\end{proof}

\begin{proposition}\label{isic} Let $M = M_r(I)$ for an integrally closed ${\mathfrak m}$-primary ideal $I$ with $ord(I) = n \geq r$.
Then, $M$ is integrally closed.
\end{proposition}

\begin{proof} We will appeal to the determinantal criterion to show that $M$ is integrally closed. Consider a vector
$$
v=
\left[
\begin{array}{c}
v_1\\v_2\\ \vdots\\v_r
\end{array}
\right] \in F
$$
that is integral over $M$. We need to see that $v \in M$. By the determinantal criterion applied to this vector and the last $r-1$ columns of $M$, $$v_1x^{r-1} +v_2x^{r-2}y + \cdots +v_ry^{r-1} \in \overline{I(M)} = I,$$ by Proposition \ref{fittingideals}. Hence there exist $c_1,\cdots,c_{n+1} \in R$ such that
$$
\left[
\begin{array}{cccc}
x^{r-1} & x^{r-2}y & \cdots & y^{r-1}
\end{array}
\right]
\left[
\begin{array}{c}
v_1\\v_2\\ \vdots\\v_r
\end{array}
\right]=
$$
$$
\left[
\begin{array}{cccc}
x^{r-1} & x^{r-2}y & \cdots & y^{r-1}
\end{array}
\right]
\left[
\begin{array}{cccc}
a_{1,1} & a_{1,2}  & \cdots & a_{1,n+1} \\
a_{2,1} & a_{2,2}  & \cdots & a_{2,n+1} \\
\vdots & \vdots &  & \vdots   \\
a_{r-1,1} & a_{r-1,2} & \cdots & a_{r-1,n+1} \\
a_{r,1} & a_{r,2} & \cdots & a_{r,n+1}
\end{array}
\right]
\left[
\begin{array}{c}
c_1\\c_2\\ \vdots\\ \vdots \\c_{n+1}
\end{array}
\right].
$$
This is because the product of the first two matrices on the right hand side gives a row vector with entries being a generating set of $I$. Hence
$$
\left[
\begin{array}{c}
v_1\\v_2\\ \vdots\\v_r
\end{array}
\right] 
-
\left[
\begin{array}{cccc}
a_{1,1} & a_{1,2}  & \cdots & a_{1,n+1} \\
a_{2,1} & a_{2,2}  & \cdots & a_{2,n+1} \\
\vdots & \vdots &  & \vdots   \\
a_{r-1,1} & a_{r-1,2} & \cdots & a_{r-1,n+1} \\
a_{r,1} & a_{r,2} & \cdots & a_{r,n+1}
\end{array}
\right]
\left[
\begin{array}{c}
c_1\\c_2\\ \vdots\\ \vdots \\c_{n+1}
\end{array}
\right]
$$
gives relations between the $x^{r-1},\cdots,y^{r-1}$ and is therefore in the span of the last $r-1$ columns of $M$. It follows that $v \in M$, as desired.
\end{proof}

\begin{example}
{\rm 
Let $I=(x^2,y) \overline{(x^3,y^2)}=(x^5,x^3y,x^2y^2,y^3)$. 
Then 
\begin{enumerate}
\item $M_2(I)$ is the module generated by the columns of matrices such as  
$$\begin{bmatrix} 
x^4&x^2y&xy^2&0&y\\
0&0&0&y^2&-x
\end{bmatrix} \ \text{or} \ 
\begin{bmatrix} 
x^4&x^2y&0&0&y\\
0&0&x^2y&y^2&-x
\end{bmatrix}. $$
\item $M_3(I)$ is the module generated by the columns of matrices such as 
$$\begin{bmatrix}
x^3&xy&0&0&y&0\\
0&0&xy&0&-x&y\\
0&0&0&y&0&-x
\end{bmatrix} \ \text{or} \
\begin{bmatrix}
x^3&0&0&0&y&0\\
0&x^2&0&0&-x&y\\
0&0&x^2&y&0&-x
\end{bmatrix}. $$
\end{enumerate}
}
\end{example}

\begin{example}
{\rm 
Let $I=(x^2,y) \overline{(x^2,y^3)}=(x^4,x^2y,xy^3,y^4)$. Then 
\begin{enumerate}
\item $M_2(I)$ is the module generated by the columns of matrices such as  
$$\begin{bmatrix} x^3&xy&y^3&0&y\\
0&0&0&y^3&-x 
\end{bmatrix} \ \text{or} \ 
\begin{bmatrix} x^3&0&0&0&y\\
0&x^2&xy^2&y^3&-x 
\end{bmatrix}. $$
\item $M_3(I)$ is the module generated by the columns of matrices such as 
$$\begin{bmatrix}
x^2&y&0&0&y&0\\
0&0&y^2&0&-x&y\\
0&0&0&y^2&0&-x
\end{bmatrix} \ \text{or} \
\begin{bmatrix}
x^2&0&0&0&y&0\\
0&x&0&0&-x&y\\
0&0&xy&y^2&0&-x
\end{bmatrix}. $$
\end{enumerate}
}
\end{example}

\smallskip

\section{Existence results}\label{exissec}

Before going on to prove existence results for general rank $r$, we will need to recall various facts. 
Recall that a parameter module $P$ of rank $r$ is a module generated by the columns of an $r \times (r+1)$ matrix with all entries in ${\mathfrak m}$ and ${\mathfrak m}$-primary ideal of maximal minors $I(P)$.
Dualising the Hilbert-Burch resolution
\begin{equation*}
\xymatrix@C+2pc{
0 \ar[r]
&
  R^{r} \ar[r]^{P^T}
  &
  R^{r+1} \ar[r]^{I(P)}
  &
  R
   \ar[r]
  & 
  0,
}
\end{equation*}
of $\frac{R}{I(P)}$ gives a resolution of $Ext^2(\frac{R}{I(P)},R)$,  which, by local duality - see Theorem 3.5.8 of \cite{BrnHrz1993} - is the Matlis dual of $\frac{R}{I(P)}$ and hence has the same length. 
The notation $\lambda(\cdot)$ stands for the length function on $R$-modules.
Thus with $F=R^{r}$, we have that $\lambda(\frac{F}{P}) =\lambda(\frac{R}{I(P)})$. Every torsion-free module $M$ without free direct summands contains a parameter submodule $P$ over which it is integral (or equivalently $P$ is a minimal reduction of $M$) - see Lemma 2.2 of \cite{Res1987}.

We next recall a numerical identity for integrally closed modules which is the content of Corollary 4.3 of \cite{KdyMhn2015} and which we will refer to in the sequel as the length-multiplicity identity. The notation $e(M)$ in this identity stands for the Buchsbaum-Rim multiplicity of $M$ - see Theorem 3.1 of \cite{BchRim1964} -
which can be computed as $\lambda(\frac{F}{P})$ for any minimal reduction $P$ of $M$. The identity states that for integrally closed modules $M$ with ideal of minors $I$,
$$
\lambda(\frac{R}{I}) - \lambda(\frac{F}{M}) = e(I) - e(M).
$$

The third fact we recall is a certain inequality for integrally closed modules without free direct summands that occurs as 
Theorem 4.4 of \cite{Hys2022}. We give a different proof of this here.

\begin{proposition}\label{ineq} Let $M$ be an integrally closed module of rank $r$ with ideal of minors $I$ and without free direct summands. Then,
$$
\lambda(\frac{R}{I}) - \lambda(\frac{F}{M}) \geq \binom{r}{2}.
$$
\end{proposition}

The proof of this proposition uses a general position argument that we will use again later in this paper.

\begin{proposition}\label{genposi} Let $M$ be a non-free, torsion-free 
$R$-module of rank $r$. 
There is a minimal reduction $N$ of $M$ such that if $N^T$ resolves the ideal $J = (a_1,\cdots,a_{r+1})$, then $a_1,a_2$ form a reduction of $J$.
\end{proposition}

\begin{proof} 
When $M$ has a free summand, write 
$M=G\oplus M'$ where $G$ is free and $M'$ has no free direct summands.  
Note that $F=G \oplus F'$ where $F'$ is the double $R$-duals of $M'$ and if 
$N'$ is a minimal reduction of $M'$, then $N=G\oplus N'$ is a minimal reduction of $M$ with $I(N)=I(N')$. Hence, we may assume that $M=M'$ has no free direct summands.

Let $P$ be a reduction of $M$ generated by $r+1$ elements and suppose that the signed minors of $P$
are $b_1,\cdots,b_{r+1}$ so that $P^T$ resolves the ideal $(b_1,\cdots,b_{r+1})$ minimally.

The ideal $J=(b_1,\cdots,b_{r+1})$ has a 2-generated reduction say $a_1,a_2$ which can be extended to a minimal generating set $(a_1,a_2,\cdots,a_{r+1})$ of the same ideal. There is a matrix $Q \in GL_{r+1}(R)$ such that
$$
B =\left[
\begin{array}{cccc}
b_1 & b_2 & \cdots & b_{r+1}
\end{array}
\right]
= 
\left[
\begin{array}{cccc}
a_1 & a_2 & \cdots & a_{r+1}
\end{array}
\right] Q = AQ.
$$
Since the sequence
\begin{equation*}
\xymatrix@C+2pc{
0 \ar[r]
&
  R^r \ar[r]^{P^T}
  &
  R^{r+1} \ar[r]^{AQ}
  &
  R
   \ar[r]
  & 
  \frac{R}{J}
   \ar[r]
  & 
  0
}
\end{equation*}
is exact, so is the sequence
\begin{equation*}
\xymatrix@C+2pc{
0 \ar[r]
&
  R^r \ar[r]^{QP^T}
  &
  R^{r+1} \ar[r]^{A}
  &
  R
   \ar[r]
  & 
  \frac{R}{J}
   \ar[r]
  & 
  0
}
\end{equation*}

So the signed minors of $QP^T$ are the $a_1,a_2,\cdots,a_{r+1}$ up to a unit of $R$ by the Hilbert-Burch theorem. Let $N$ be $PQ^T$ so that the columns of $P$ and $N$ generate the same submodule. Thus $N$ is a minimal reduction of $M$ with $N^T = QP^T$  resolving the ideal $(a_1,\cdots,a_{r+1})$  with $a_1,a_2$ forming a minimal reduction of it, as needed.
\end{proof}

\begin{proof}[Proof of Proposition \ref{ineq}] By Proposition \ref{genposi}, choose a minimal reduction $N$ of $M$ such that if $N^T$ resolves the ideal $J = (a_1,\cdots,a_{r+1})$, then $a_1,a_2$ form a reduction of $J$. Then,
$$
\lambda(\frac{R}{I}) - \lambda(\frac{F}{M}) = e(I) - e(M)  = e(J) -  \lambda(\frac{F}{N}) = e(J) -  \lambda(\frac{R}{J}),
$$
where the first equality is the length-multiplicity identity, the second follows from the determinantal criterion since $N$ is a reduction of $M$ and the last from properties of parameter modules.

If $B$ is the submatrix of $N$ of the last $r-1$ columns, then standard homological algebra comparing the Koszul resolution of $a_1,a_2$ with that of the resolution of $J$ shows that $B^T$ resolves the quotient module $\frac{J}{(a_1,a_2)}$. Since $B^T$ is also a parameter matrix, $\lambda(\frac{J}{(a_1,a_2)}) = \lambda(\frac{R}{I(B)})$. But all entries of $M$ hence of $N$ and of $B$ are in ${\mathfrak m}$, so that $I(B) \subseteq {\mathfrak m}^{r-1}$. So finally,
$$
e(J) -  \lambda(\frac{R}{J}) = \lambda(\frac{R}{(a_1,a_2)}) -  \lambda(\frac{R}{J}) =\lambda(\frac{J}{(a_1,a_2)}) =\lambda(\frac{R}{I(B)}) \geq \lambda(\frac{R}{{\mathfrak m}^{r-1}}) = \binom{r}{2}.
$$
\end{proof}

Motivated by Proposition \ref{ineq}  and following the general philosophy that understanding the structure of extremal cases of inequalities is useful, we make the following definition.

\begin{definition}
An integrally closed $R$-module $M$ of rank $r$ with ideal of minors $I$ and without free direct summands is said to be extremal if
$$
\lambda(\frac{R}{I}) - \lambda(\frac{F}{M}) = \binom{r}{2}.
$$
\end{definition}

We will apply the next proposition in the proof of extremality of the modules $M_r(I)$.

\begin{proposition}\label{loc}
Let $(a,b)$ be an ${\mathfrak m}$-primary ideal of $R$ of order $n$ and let $2 \leq r \leq n$. Write $a = a_1x^{r-1} + a_2x^{r-2}y + \cdots + a_ry^{r-1}$ and similarly for $b$. Let $P$ be the parameter module generated by the columns of the matrix
$$
\left[
\begin{array}{ccrrrr}
a_1 & b_{1}    & y & 0 &   0 & 0\\
a_2 & b_{2}    & -x & y &   0 & 0\\
\vdots & \vdots &   & \ddots  & \ddots & \vdots\\
a_{r-1} & b_{r-1} &    0 & 0   & -x & y\\
a_{r} & b_{r} &   0 & 0   & 0 & -x
\end{array}
\right]. 
$$
Then $(a, b) \subseteq I(P)$ and 
$$\lambda \left(\frac{I(P)}{(a,b)}\right) = \binom{r}{2}.$$
\end{proposition}

\begin{proof}
The transposed matrix $P^T$ resolves the ideal $I(P) = ((-1)^{r-1}b, (-1)^ra,c_1,\cdots,c_{r-1})$ for some $c_1,\cdots,c_{r-1}$. 
Comparing with the Koszul resolution of $((-1)^{r-1}b, (-1)^ra)$ as in the proof of Proposition \ref{ineq},  the quotient module
$\frac{I(P)}{(a,b)}$ is
resolved by the $ (r-1) \times r$ matrix
$$
\left[
\begin{array}{ccccc}
  y & -x &   0 & \cdots & 0\\
 0 & y &   -x &  \cdots & 0\\
\vdots & \vdots & \ddots  & \ddots  & \vdots\\
0 & 0   & \cdots &y & -x\\
0 & 0   & \cdots & 0 & y
\end{array}
\right]
$$
obtained by dropping the first two rows of the transposed matrix.
The ideal of minors of this parameter matrix is ${\mathfrak m}^{r-1}$. So the quotient is isomorphic to $Ext^2(\frac{R}{{\mathfrak m}^{r-1}},R)$ and hence has length $\binom{r}{2}$.
\end{proof}

We can now give a large class of examples of extremal modules.

\begin{proposition}\label{isextr} Let $M = M_r(I)$ for an integrally closed ${\mathfrak m}$-primary ideal $I$ with $ord(I) = n \geq r$.
Then, $M$ is extremal.
\end{proposition}

\begin{proof} First, $M$ is integrally closed by Proposition \ref{isic}. Next, by construction $M \subseteq {\mathfrak m}F$ and hence does not have free direct summands. Then, by the length-multiplicity identity stated above, 
it suffices to see that $e(I) - e(M) = \binom{r}{2}$ to see that $M$ is extremal.

Let $a_1,a_2$ be a minimal reduction of $I$ and extend it to a minimal generating set $a_1,\cdots,a_{n+1}$ of $I$.
Consider $M=M_r(I)$ constructed from this generating set.
Let $P$ be the submodule of $M$ generated by the first two columns - that correspond to the generators $a_1,a_2$ of $I$ - together with the last $r-1$ columns. Then in $I(P)$ we have $a_1$ and $a_2$ so $I(P)$ is a reduction of $I$ and so $P$ is a reduction of $M$ by the determinantal criterion.

Hence, $e(M) = \lambda(\frac{F}{P}) = \lambda(\frac{R}{I(P)})$ while $e(I) = \lambda(\frac{R}{(a_1,a_2)})$. Therefore $e(I) - e(M)$ equals $\lambda(\frac{I(P)}{(a_1,a_2)}) = \binom{r}{2}$, by Proposition \ref{loc}.
\end{proof}

Rather interestingly, it turns out that Proposition \ref{isextr} not only gives a large class of examples of extremal modules, it actually gives all of them. Equivalently, the $M_r(I)$ are the only examples of extremal modules. This is a numerical-structural type result of the type $M$ is contracted iff $\mu(M) = ord(M)+rk(M)$, in which this theory abounds.

\begin{proposition} \label{extrcharac}
Let $M$ be an extremal module of rank $r$ with ideal of minors $I$. Then $M$ is isomorphic to $M_r(I)$.
In particular $I_k(M) = {\mathfrak m}^{k}$ for $k = 1,2,\cdots,r-1$.
\end{proposition}

\begin{proof} As in the proof of Proposition \ref{ineq}, choose a minimal reduction $N$ of $M$ such that if $N^T$ resolves the ideal $J = (a_1,a_2,\cdots,a_{r+1})$, then $a_1,a_2$ form a reduction of $J$. Since $M \subseteq {\mathfrak m}F$, $J$ is minimally $r+1$ generated. Now we're given that
$$
\lambda(\frac{R}{I}) - \lambda(\frac{F}{M}) = e(I) - e(M) = e(J) - \lambda(\frac{F}{N}) = e(J) - \lambda(\frac{R}{J}) = \binom{r}{2}.
$$
Let $B$ be the submatrix of $N$ of the last $r-1$ columns so that $B^T$ resolves $\frac{J}{(a_1,a_2)}$
whose length is $\binom{r}{2}$. So $\lambda(\frac{R}{I(B)}) = \binom{r}{2}$. But since $I(B) \subseteq {\mathfrak m}^{r-1}$, equality must hold and $I(B) = {\mathfrak m}^{r-1}$. So $B$ resolves ${\mathfrak m}^{r-1}$ minimally. Now, by premultiplying $B$ by a suitable matrix $Q \in GL_r(R)$ as in  Proposition \ref{genposi}, we may get it to the form
$$
\left[
\begin{array}{rrrr}
y & 0 &   0 & 0\\
-x & y &   0 & 0\\
 \vdots & \ddots  & \ddots & \vdots\\
 0 & 0   & -x & y\\
0 & 0   & 0 & -x
\end{array}
\right].
$$
Premultiplying  $N$ by $Q$ - which does not change the isomorphism class of $M= \overline{N}$ -  we may assume that $N$ has the form
$$
\left[
\begin{array}{ccrrrr}
a_{1,1} & a_{1,2}    & y & 0 &   0 & 0\\
a_{2,1} & a_{2,2}    & -x & y &   0 & 0\\
\vdots & \vdots &   & \ddots  & \ddots & \vdots\\
a_{r-1,1} & a_{r-1,2} &    0 & 0   & -x & y\\
a_{r,1} & a_{r,2} &   0 & 0   & 0 & -x
\end{array}
\right].
$$
Note that the first two minors have only changed up to a multiplicative unit - $det(Q)$.  So we may assume that the first two minors are still $a_1,a_2$.
Now $M$ is the integral closure of the matrix above.

Extend $a_1,a_2$ to a minimal generating set of $I$ say $a_1,\cdots,a_{n+1}$ and consider $M_r(I)$ constructed from these generators with the first two being lifted as above. Then $M_r(I)$ is integrally closed by Proposition \ref{isic} and is integral over $N$ since the ideal of minors of $M_r(I)$ which is $I$ is integral over $(a_1,a_2) \in I(N)$. Thus $M_r(I)$ is the
integral closure of $N$ and hence equals $M$. The last assertion follows from Proposition \ref{fittingideals}.
\end{proof}

One consequence of this is that, for instance, there is a unique extremal module of rank $r$ with ideal of minors $I$ 
for every integrally closed ideal $I$ of order at least $r$.
There are rather tight restrictions on how an extremal module may decompose as a direct sum which will help us prove indecomposability in many cases.

\begin{proposition}\label{tight} Let $M$ be extremal of rank $r$ with $I(M) = I$ having order $n$. Suppose that $M$ decomposes as
$P_1 \oplus P_2$  where $rk(P_i) = r_i$ and $I(P_i) = I_i$ of order $n_i$. Then,
\begin{enumerate}
\item $n_1 = r_1, n_2 = r_2$ and so, $n=r$.
\item $P_1$ and $P_2$ are extremal.
\item $\lambda(\frac{R}{I}) = \lambda(\frac{R}{I_1}) + \lambda(\frac{R}{I_2}) + r_1r_2$.
\item ${\mathfrak m}^{r-1} = {\mathfrak m}^{r_2-1}I_1 + {\mathfrak m}^{r_1-1}I_2.$
\end{enumerate}
\end{proposition}

We note that the condition (3) in the above proposition is also an extremal case of an inequality. For any two integrally closed ${\mathfrak m}$-primary ideals $I_1$ and $I_2$ with $ord(I_1) = n_1, ord(I_2) = n_2$, the Hoskin-Deligne length formula
- see Theorem 3.1 of \cite{Lpm1988} for instance - 
implies that
$$\lambda(\frac{R}{I_1I_2}) - \lambda(\frac{R}{I_1}) - \lambda(\frac{R}{I_2}) = \sum_{T} ord_T(I_1^T)ord_T(I_2^T)[T:R] = n_1n_2 + \cdots
$$ where the sum is over all quadratic transforms $T$ of $R$ and $[T:R]$ denotes the degree of the residue field extension of $T$ over $R$. The $n_1n_2$ term comes from $T = R$. So for condition (3) to hold,
 $I_1$ and $I_2$ must not have a common base-point other than $R$ itself. This is a
necessary and sufficient condition and is equivalent to saying that there is no first quadratic transform $T$ of $R$ in which both transforms
$I_1^T$ and $I_2^T$ are proper ideals.

\begin{proof}[Proof of Proposition \ref{tight}] Note that $r = r_1+r_2$, $I=I_1I_2$ and $n = n_1+n_2$. Also, because
$M$ does not have a free direct summand, neither do $P_1$ and $P_2$, and hence $n_1 \geq r_1$ and $n_2 \geq r_2$. Let $F_i = P_i^{**}$ so that $M^{**} = F = F_1 \oplus F_2$. By extremality of $M$,
$$
\lambda(\frac{R}{I}) - \lambda(\frac{F}{M}) - \binom{r}{2} = 0.
$$
Hence,
$$
\lambda(\frac{R}{I_1I_2}) - \lambda(\frac{F_1}{P_1}) - \lambda(\frac{F_2}{P_2}) - \binom{r_1+r_2}{2}  = 0.
$$
Rewrite this as:
\begin{eqnarray*}
\left\{  \lambda(\frac{R}{I_1I_2}) - \lambda(\frac{R}{I_1}) - \lambda(\frac{R}{I_2}) - n_1n_2 \right\} &+& \left\{  \lambda(\frac{R}{I_1}) - \lambda(\frac{F_1}{P_1}) -  \binom{r_1}{2}  \right\}\\
 + \left\{ \lambda(\frac{R}{I_2}) - \lambda(\frac{F_2}{P_2}) -  \binom{r_2}{2} \right\} &+& \left\{n_1n_2 - r_1r_2\right\} = 0
\end{eqnarray*}
Since all the bracketed terms are non-negative, they all must vanish. So we deduce
\begin{enumerate}
\item $n_1n_2 = r_1r_2$ which implies that $n_1 = r_1, n_2 = r_2$ and $n = n_1+n_2 = r_1+r_2 = r$.
\item $P_1$ and $P_2$ are extremal from the vanishing of the second and third bracketed terms.
\item $\lambda(\frac{R}{I_1I_2}) = \lambda(\frac{R}{I_1}) + \lambda(\frac{R}{I_2}) + n_1n_2  = \lambda(\frac{R}{I_1}) + \lambda(\frac{R}{I_2}) + r_1r_2$.
\item Since $M$ decomposes as $P_1 \oplus P_2$, 
\begin{eqnarray*}
I_{r-1}(M) &=& I_{r_2-1}(P_2)I_{r_1}(P_1)+I_{r_1-1}(P_1)I_{r_2}(P_2) {\text {~and so~}}\\
{\mathfrak m}^{r-1} &=& {\mathfrak m}^{r_2-1}I_1 + {\mathfrak m}^{r_1-1}I_2,
\end{eqnarray*}
\end{enumerate}
where the last equation follows using Proposition \ref{extrcharac}.
\end{proof}

Our main existence theorem now follows easily.

\begin{theorem}\label{existence} Let $I$ be an integrally closed ${\mathfrak m}$-primary ideal of order $n$. If one of the following two conditions hold, 
\begin{enumerate}
\item $n > r$.
\item $n=r$ and there do not exist ${\mathfrak m}$-primary integrally closed ideals $I_1$ and $I_2$ of orders $r_1$ and $r_2$ respectively such that
$I= I_1I_2$ and both the conditions below hold:
\begin{enumerate}
\item $\lambda(\frac{R}{I_1I_2}) = \lambda(\frac{R}{I_1}) + \lambda(\frac{R}{I_2}) + r_1r_2$, and
\item ${\mathfrak m}^{r-1} = {\mathfrak m}^{r_2-1}I_1 + {\mathfrak m}^{r_1-1}I_2$,
\end{enumerate}
\end{enumerate}
then $I=I(M)$ for an indecomposable integrally closed module $M$ of rank $r$.
\end{theorem}

\begin{proof} 
Let $M = M_r(I)$ which is an extremal module by Proposition \ref{isextr} and satisfying $I(M) = I$ by Proposition \ref{fittingideals}. If $M$ decomposes as $P_1 \oplus P_2$ where $rk(P_i) = r_i$ and $I(P_i) = I_i$ of order $n_i$, then Proposition \ref{tight} implies that $n=r$, $I=I_1I_2$ with $n_1 = r_1, n_2 = r_2$ and both the conditions (a) and (b) of (2) hold. It follows that if (1) or  (2) holds for $I$, then $M$ is necessarily indecomposable.
\end{proof}

\begin{example}
{\rm 
Let $I=(x,y^{m_1})(x,y^{m_2}) \cdots (x,y^{m_r})$ where $2 \leq m_1 \leq m_2 \leq \cdots \leq m_r$ and $r \geq 2$. 
Let $I_1$ and $I_2$ be $\fkm$-primary integrally closed ideals of orders $r_1$ and $r_2$ 
respectively such that $I=I_1I_2$. Note that $I_1$ and $I_2$ are also monomial ideals 
by Zariski's unique factorisation theorem. 
It is easy to see that $y^{r-1} \notin \fkm^{r_2-1} I_1+\fkm^{r_1-1} I_2$, 
so the condition (2)(b) in Theorem \ref{existence} is not satisfied. 
Hence, $I=I(M)$ for an indecomposable integrally closed module $M$ of rank $r$. 
}
\end{example}

\section{Non-existence results}\label{nonexissec}

The rest of the paper deals with non-existence results which will yield a converse to Theorem \ref{existence} in the rank 3 case. We begin by analysing the conditions (a) and (b) of Theorem \ref{existence}(2) in terms of joint reductions - see \cite{Res1984} for the definition.
We will need the main result of \cite{Vrm1990} which asserts that for integrally closed ${\mathfrak m}$-primary ideals $J$ and $K$ of $R$, there exist $a \in J, b \in K$ such that $JK = aK+bJ$ and that in fact, this holds for any joint reduction
$a,b$ of $J,K$. Our first result may be regarded as a more structural characterisation of the numerical equality in
Theorem \ref{existence}(2)(a)

\begin{proposition}\label{jtred}
Let $J,K$ be ${\mathfrak m}$-primary integrally closed  ideals of $R$ of orders $s$ and $t$ respectively and let $I = JK$ with order $r=s+t$.
The following two conditions are equivalent:
\begin{enumerate}
\item The conditions below hold for $J$ and $K$.
\begin{enumerate}
\item $\lambda(\frac{R}{JK}) = \lambda(\frac{R}{J}) + \lambda(\frac{R}{K}) + st$, and
\item ${\mathfrak m}^{r-1} = {\mathfrak m}^{t-1}J + {\mathfrak m}^{s-1}K$,
\end{enumerate}
\item For any joint reduction $a,b$ of $J,K$, 
\begin{enumerate}
\item $JK = aK + bJ$, and
\item  ${\mathfrak m}^{r-1} = a{\mathfrak m}^{t-1} + b{\mathfrak m}^{s-1}$.
\end{enumerate}
\end{enumerate}
\end{proposition}

\begin{proof} (1) $\Rightarrow$ (2): Let $a,b$ be a joint reduction of $J,K$ so that $JK = aK+bJ$ necessarily holds. Next, consider the map
$$
\frac{{\mathfrak m}^s}{J} \oplus \frac{{\mathfrak m}^t}{K} \longrightarrow \frac{{\mathfrak m}^r}{I},
$$
defined by $(\overline{c},\overline{d}) \mapsto \overline{bc+ad}$. This is clearly well-defined and further it is injective since $JK = aK+bJ$ and $a,b$ form a regular sequence. The numerical equality in (1)(a) is equivalent to this map being an isomorphism (or surjective). Thus, ${\mathfrak m}^r =  b{\mathfrak m}^{s}+ a{\mathfrak m}^{t} + JK = b{\mathfrak m}^{s}+ a{\mathfrak m}^{t}$ since $JK = aK+bJ$.

We now claim that together with condition (1)(b) this implies that ${\mathfrak m}^{r-1} = b{\mathfrak m}^{s-1} + a{\mathfrak m}^{t-1}$.
To see this, take $z \in {\mathfrak m}^{r-1}$ and write it using (1)(b) as $mj+nk$ with $m \in {\mathfrak m}^{t-1}$, $n \in {\mathfrak m}^{s-1}$, $j \in J$, $k \in K$. Now $zx,zy \in {\mathfrak m}^{r}$ and so we have equations
\begin{eqnarray*}
zx&=&am^{\prime} + bn^{\prime},\\
zy&=&am^{\prime\prime} + bn^{\prime\prime},
\end{eqnarray*}
with $m^{\prime},m^{\prime\prime} \in {\mathfrak m}^{t}$ and $n^{\prime},n^{\prime\prime} \in {\mathfrak m}^{s}$, since ${\mathfrak m}^r =   b{\mathfrak m}^{s}+ a{\mathfrak m}^{t}$. Therefore
$y(am^{\prime} + bn^{\prime}) = x(am^{\prime\prime} + bn^{\prime\prime})$, or equivalently,
$a(ym^{\prime}-xm^{\prime\prime}) = b(xn^{\prime\prime}-yn^{\prime})$. Since $a,b$ form a regular sequence, there is a $u \in R$ such that 
\begin{eqnarray*}
au &=& xn^{\prime\prime}-yn^{\prime}, {\text {~and~}}\\
bu&=&ym^{\prime}-xm^{\prime\prime}.
\end{eqnarray*}
The equation ${\mathfrak m}^r  = b{\mathfrak m}^{s}+ a{\mathfrak m}^{t}$ implies that  $ord(a) = s$ and $ord(b) =t$. Hence $u$ is not a unit. Write $u=px+qy$ for $p,q \in R$. The equations above then imply that
\begin{eqnarray*}
y(n^{\prime}+aq) &=& x(n^{\prime\prime}-ap), {\text {~and~}}\\
x(m^{\prime\prime} + bp) &=&y(m^{\prime}-bq).
\end{eqnarray*}
Hence there exist $v,w \in R$ such that $n^{\prime}+aq=vx, n^{\prime\prime}-ap=vy,m^{\prime\prime} + bp=wy, m^{\prime}-bq=wx$.

Finally, $zx = am^{\prime} + bn^{\prime} = a(m^{\prime}-bq)+b(n^{\prime}+aq) = awx+bvx$, and
so $z = aw+bv$. The equations $m^{\prime}-bq=wx$ and $n^{\prime\prime}-ap=vy$ imply that
$w \in {\mathfrak m}^{t-1}, v \in {\mathfrak m}^{s-1}$ as required.\\
(2) $\Rightarrow$ (1): Clearly (2)(b) implies (1)(b). As for (1)(a), we have observed above that it is equivalent to
${\mathfrak m}^r  = b{\mathfrak m}^{s}+ a{\mathfrak m}^{t}$, given (2)(a), and this is also implied by (2)(b).
\end{proof}

We will need the next two lemmas in the proof of our non-existence results.

\begin{lemma}\label{msummand} Let $M \subseteq F = M^{**}$ be a torsion-free $R$-module. Then $I_1(M)$ is a direct summand of $M$ iff $(M:_FI_1(M))$ contains a minimal generator of $F$.
\end{lemma}

\begin{proof}
Suppose that $v \in F$ is a minimal generator of $F$ and $I_1(M)v \subseteq M$. There is a row vector $u$ such that $u.v = 1$. Consider the homomorphisms
\begin{equation*}
\xymatrix@C+2pc{
  R \ar[r]^{v}
  &
  F \ar[r]^{u}
  &
  R
}
\end{equation*}
Restricting the first homomorphism to $I_1(M) \subseteq R$ gives a homomorphism from $I_1(M)$ to $M$ since $I_1(M)v \subseteq M$.
Restricting the second homomorphism to $M \subseteq F$ gives a homomorphism from $M$ to $I_1(M)$. Since the composite homomorphism $u.v$ is the identity map of $R$ it follows that the composition of the restricted maps is the identity map of $I_1(M)$. This implies that $I_1(M)$ is a direct summand of $M$. Conversely, if $M = I_1(M) \oplus N$, then, $F = R \oplus N^{**}$ and the generator of $F$ that corresponds to the $R$ direct summand is in $(M:_FI_1(M))$.
\end{proof}

\begin{lemma}\label{contr}
If $M$ is an integrally closed module with ideal of minors $I$ and $I$ is contracted from  $R[\frac{\mathfrak m}{y}]$, then, so is $M$.
\end{lemma}

\begin{proof}
We know that $(I:y) = (I:{\mathfrak m})$ and need to see that $(M:_F y) = (M:_F {\mathfrak m})$. So suppose that $v \in F$ is such that $yv \in M$. By the determinantal criterion, for any $r-1$ elements $m_2,\cdots,m_r$ of $M$, 
the determinant of $[yv ~ m_2 ~ \cdots ~ m_r] \in \overline{I} = I$. And hence 
$det ([{\mathfrak m}v ~ m_2 ~ \cdots ~ m_r]) \subseteq I$. So ${\mathfrak m}v \subseteq \overline{M} = M$.
\end{proof}

We are now prepared to prove our first non-existence result.

\begin{proposition}\label{nonexis1} Suppose that $I$ is integrally closed of order $r$ and the image of $I$ in $\frac {{\mathfrak m}^r}{{\mathfrak m}^{r+1}}$ is at least 2-dimensional. If $I$ is the ideal of minors of an integrally closed module $M$ of rank $r$ without free direct summands, then ${\mathfrak m}$ is a direct summand of $M$.
\end{proposition}

\begin{proof} Choose any minimal generator $y$ of ${\mathfrak m}$ such that $I$ is contracted from $R[\frac{\mathfrak m}{y}]$. Then it is well known - see \cite{Kdy1995} for instance - that
$$\lambda(\frac{R}{(I,y)}) = \lambda(\frac{I:y}{I}) = \lambda(\frac{I:\mathfrak m}{I}) = ord(I) = r.$$
So if ${\mathfrak m} = (x,y)$, then $(I,y) = (x^r,y)$. This implies that $x^r + a_1x^{r-1}y+a_2x^{r-2}y^2 + \cdots + a_ry^r \in I$ for some $a_1,\cdots,a_r$.

Suppose that $I$ is the ideal of minors of an integrally closed module $M$ of rank $r$ without free direct summands. Since $I$ is of order $r$ and $rk(M) = r$,  $M$ is minimally $2r$-generated. Since $M$ has no free direct summands we may assume that $M$ is generated by the columns of a matrix of the form:
$$
\left[
\begin{array}{cccc}
a_{1,1}x+b_{1,1}y & a_{1,2}x+b_{1,2}y & \cdots & a_{1,2r}x+b_{1,2r}y\\
a_{2,1}x+b_{2,1}y & a_{2,2}x+b_{2,2}y & \cdots & a_{2,2r}x+b_{2,2r}y\\
\vdots & \vdots & & \vdots \\
a_{r,1}x+b_{r,1}y & a_{r,2}x+b_{r,2}y & \cdots & a_{r,2r}x+b_{r,2r}y
\end{array}
\right].
$$

Every maximal minor of this matrix may be regarded as a homogeneous form of degree $r$ in $x,y$ with coefficients
being polynomials in the $a_{i,j}$ and $b_{i,j}$. Since $x^r + a_1x^{r-1}y+a_2x^{r-2}y^2 + \cdots + a_ry^r$ is in the ideal generated by these forms, the coefficient
of $x^r$ in at least one of these minors must be a unit. By column operations, we may assume that it is the minor corresponding to the
first $r$ columns. The coefficient of $x^r$ in this minor is the determinant of the matrix
$$\left[
\begin{array}{cccc}
a_{1,1} & a_{1,2} & \cdots & a_{1,r}\\
a_{2,1} & a_{2,2} & \cdots & a_{2,r}\\
\vdots & \vdots & & \vdots \\
a_{r,1} & a_{r,2} & \cdots & a_{r,r}
\end{array}
\right].
$$
So this matrix is in $GL(r,R)$ and multiplying by its inverse, we may assume that $M$ is generated by a matrix of the form
$$
\left[
\begin{array}{ccccccc}
x+b_{1,1}y & b_{1,2}y & \cdots & b_{1,r}y & a_{1,r+1}x+b_{1,r+1}y & \cdots &a_{1,2r}x+b_{1,2r}y\\
b_{2,1}y & x+b_{2,2}y & \cdots & b_{2,r}y & a_{2,r+1}x+b_{2,r+1}y & \cdots &a_{2,2r}x+b_{2,2r}y\\
\vdots & \vdots & \ddots & \vdots & \vdots & \ddots & \vdots\\
b_{r,1}y & b_{r,2}y & \cdots & x+b_{r,r}y & a_{r,r+1}x+b_{r,r+1}y & \cdots &a_{r,2r}x+b_{r,2r}y
\end{array}
\right],
$$
and then by further column operations, by a matrix of the form
$$
\left[
\begin{array}{ccccccc}
x+b_{1,1}y & b_{1,2}y & \cdots & b_{1,r}y & b_{1,r+1}y & \cdots & b_{1,2r}y\\
b_{2,1}y & x+b_{2,2}y & \cdots & b_{2,r}y & b_{2,r+1}y & \cdots &b_{2,2r}y\\
\vdots & \vdots & \ddots & \vdots & \vdots & \ddots & \vdots\\
b_{r,1}y & b_{r,2}y & \cdots & x+b_{r,r}y & b_{r,r+1}y & \cdots & b_{r,2r}y
\end{array}
\right].
$$

If none of the $b_{i,r+k}$ are units, then the last $r$ columns have all entries in ${\mathfrak m}^2$ and so the only generator of $I$ of order $r$ would be the $[1 ~ 2 ~ \cdots ~ r]$-minor. Hence the image of $I$ in $\frac {{\mathfrak m}^r}{{\mathfrak m}^{r+1}}$ would be 1-dimensional contradicting the hypothesis on $I$. 
If $v \in F$ is the column of coefficients of $y$ where some $b_{i,r+k}$ is a unit, then $v$ is a minimal generator of $F$ and $yv \in M$. Further, $I_1(M) = {\mathfrak m}$.
Thus, $v \in (M:_Fy) = (M:_F{\mathfrak m})$, since $I$ and therefore also $M$ - by Lemma \ref{contr} - is contracted from $R[\frac{\mathfrak m}{y}]$. Now by Lemma \ref{msummand}, $\mathfrak m$ is a direct summand of $M$.
\end{proof}

\begin{example}
{\rm 
Let $I=(x,y)(x,y^{m_2}) \cdots (x,y^{m_r})$ where $1 \leq m_2 \leq \cdots \leq m_r$ and $r \geq 2$.
Then $x^r, x^{r-1}y \in I$ so that the image of $I$ in $\frac{\fkm^r}{\fkm^{r+1}}$ is at least $2$-dimensional. 
By Proposition \ref{nonexis1}, 
$I$ is not $I(M)$ for an indecomposable integrally closed module $M$ of rank $r$. 
}
\end{example}

\begin{corollary}\label{splcase1} Suppose that $K$ is an integrally closed ${\mathfrak m}$-primary ideal of order $r-1$. Then if $I = {\mathfrak m}K$ is the ideal of minors of an integrally closed module $M$ of rank $r$ without free direct summands, then ${\mathfrak m}$ is a direct summand of $M$. 
\end{corollary}

\begin{proof} Choose $y$ so that $K$ is contracted from $R[\frac{\mathfrak m}{y}]$ and $x$ so that ${\mathfrak m} = (x,y)$. As in the proof of Proposition \ref{nonexis1}, $x^{r-1} + a_1x^{r-2}y+a_2x^{r-3}y^2 + \cdots + a_{r-1}y^{r-1} \in K$ for some $a_1,\cdots,a_{r-1} \in R$. Hence $x^{r} + a_1x^{r-1}y+a_2x^{r-2}y^2 + \cdots + a_{r-1}xy^{r-1}, x^{r-1}y + a_1x^{r-2}y^2+a_2x^{r-3}y^3 + \cdots + a_{r-1}y^{r} \in I$. So the image of $I$ in $\frac {{\mathfrak m}^{r}}{{\mathfrak m}^{r+1}}$ is at least 2-dimensional. Now appeal to Proposition \ref{nonexis1} to finish the proof.
\end{proof}

The next proposition is our main non-existence result. In the proof we will use the observation that if $M$ is a module of rank $r$ without free direct summands that is integrally closed and with $I(M) = I$, then $M \supseteq (I:I_1(M)^{r-1})F \supseteq (I:{\mathfrak m}^{r-1})F$. This follows from the determinantal criterion.

\begin{proposition}\label{nonexis2} Let $J$ and $K$ be integrally closed ${\mathfrak m}$-primary ideals of $R$
of orders $1$ and $2$ respectively such that 
\begin{enumerate}
\item $\lambda(R/JK)-\lambda(R/J)-\lambda(R/K)=2$, and
\item $\mathfrak m^2=J \mathfrak m + K$.
\end{enumerate}
Suppose that $M$ is an integrally closed module of rank $3$ without free direct summands and with ideal of minors $I=JK$. Then $M$ is decomposable. 
\end{proposition}

\begin{proof} Let $a,b$ be a joint reduction of $J,K$ so that $I = aK+bJ$. The element $a$ is of order $1$ and
so is a minimal generator, say $y$, of ${\mathfrak m}$. Then, by Proposition \ref{jtred}, $\mathfrak m^2=y \mathfrak m + bR$.
Choose $x$ such that $J$ and $K$ are contracted from $R[\frac{\mathfrak m}{x}]$ and ${\mathfrak m} = (x,y)$. Since $ord(b) = 2$, write $b = px^2+qxy+ry^2$. Since
$\lambda(\frac{R}{(y,b)}) = 2$, $p$ is necessarily a unit and we may assume that it is $1$ by changing $b$ appropriately. Further $J$ is of the form $(x^m,y)$ for some $m \geq 1$. Thus, $I$ contains the element $x^2y+qxy^2+ry^3$ and is also contracted from $R[\frac{\mathfrak m}{x}]$.

We will dispose of two special cases at the outset. The first is when $m=1$, so that $J = {\mathfrak m}$,  which follows from Corollary \ref{splcase1}. The second is when $x^2 \in K$, in which case $x^2y \in I$ and since $I$ is contracted from $R[\frac{\mathfrak m}{x}]$, $y^3 \in I$ so that Proposition \ref{nonexis1} applies to complete the proof. In the rest of the proof we assume that $m \geq 2$ and $x^2 \notin K$.

Note that $(I,y):x^3 = (yK,bx^m,y):x^3 = (bx^m,y):x^3 = (x^{m+2},y) : x^3 = (x^{m-1},y)$. So $x((I,y):x^3) = (x^{m},xy)$. Also, $I:{\mathfrak m}^2 = I:x^2 = (yK,bx^m):x^2$ and hence contains $bx^{m-2}=x^m+qx^{m-1}y+rx^{m-2}y^2$. So $x((I,y):x^3) \subseteq ((I:{\mathfrak m}^2),{\mathfrak m}y)$, since $m \geq 2$. Hence if $Ax^3 \in (I,y)$ then $Ax \in ((I:{\mathfrak m}^2),{\mathfrak m}y)$.

Let $M$ be an integrally closed module of rank 3 with ideal of minors $I$. Suppose that $M$ is generated by the columns of
$$
\left[
\begin{array}{cccc}
a_{11}x+b_{11}y & a_{12}x+b_{12}y & \cdots & a_{16}x+b_{16}y\\
a_{21}x+b_{21}y & a_{22}x+b_{22}y & \cdots & a_{26}x+b_{26}y\\
a_{31}x+b_{31}y & a_{32}x+b_{32}y & \cdots & a_{36}x+b_{36}y
\end{array}
\right].
$$
Since $I$ contains an element with $x^2y$ coefficient $1$, we may assume, by permuting columns, that the coefficient of $x^2y$ in the $[1~2~3]$ minor is a unit. This coefficient is the sum $|A_1~A_2~B_3| + |A_1~B_2~A_3|+|B_1~A_2~A_3|$, where $A_1$ is
$$
\left[
\begin{array}{c}
a_{11} \\ a_{21} \\ a_{31}
\end{array}
\right]
$$
and similarly the others. So at least one of those three determinants is a unit and by permuting columns we may assume that the
matrix
$$
\left[
\begin{array}{ccc}
A_1 & A_2 & B_3
\end{array}
\right]
=
\left[
\begin{array}{cccc}
a_{11} & a_{12} & b_{13}\\
a_{21} & a_{22} & b_{23}\\
a_{31} & a_{32} & b_{33}
\end{array}
\right]
$$
is invertible. Then multiply on the left by the inverse of this matrix to conclude that $M$ is generated by the columns of
$$
\left[
\begin{array}{rrlcccc}
x+b_{11}y & b_{12}y & a_{13}x & a_{14}x+b_{14}y & a_{15}x+b_{15}y & a_{16}x+b_{16}y\\
b_{21}y & x+b_{22}y & a_{23}x & a_{24}x+b_{24}y & a_{25}x+b_{25}y & a_{26}x+b_{26}y\\
b_{31}y & b_{32}y & a_{33}x+y & a_{34}x+b_{34}y & a_{35}x+b_{35}y & a_{36}x+b_{36}y
\end{array}
\right],
$$
where the coefficent of $x^2y$ in the $[1~2~3]$ minor is a unit.
Then use further column operations to get to the form
$$
\left[
\begin{array}{rrlcccc}
x+b_{11}y & b_{12}y & a_{13}x & b_{14}y & b_{15}y & b_{16}y\\
b_{21}y & x+b_{22}y & a_{23}x & b_{24}y & b_{25}y & b_{26}y\\
b_{31}y & b_{32}y & a_{33}x+y & a_{34}x+b_{34}y & a_{35}x+b_{35}y & a_{36}x+b_{36}y
\end{array}
\right],
$$
without changing the  $[1~2~3]$ minor.

Write the $[1~2~3]$ minor as $Ax^3 + Bx^2y + Cxy^2 + Dy^3$. This is in $I$ and so
$Ax^3 \in (I,y)$ and so $Ax \in ((I:{\mathfrak m}^2),{\mathfrak m}y)$. Noting that $A=a_{33}$ we see that by subtracting an element of $(I:{\mathfrak m}^2)F$, the third column  has the form
$$
\left[
\begin{array}{c}
a_{13}x\\
a_{23}x\\
(1+t )y
\end{array}
\right]
$$
for some $t \in {\mathfrak m}$. 
Similarly the 4th, 5th and 6th columns may be reduced to the form below.
So $M$ is generated modulo $(I:{\mathfrak m}^2)F$ by the columns of
$$
\left[
\begin{array}{rrlcccc}
x+b_{11}y & b_{12}y & a_{13}x & b_{14}y & b_{15}y & b_{16}y\\
b_{21}y & x+b_{22}y & a_{23}x & b_{24}y & b_{25}y & b_{26}y\\
b_{31}y & b_{32}y & y & b_{34}y & b_{35}y & b_{36}y
\end{array}
\right].
$$
Since $(I:{\mathfrak m}^2) = (yK,bx^m):x^2 = (y(K:x^2),bx^{m-2})$, this is contained in ${\mathfrak m}^2$ since $x^2 \notin K$. Now, by subtracting an element of $(I:{\mathfrak m}^2)F$,
the $[1~2~3]$ minor changes by an element of ${\mathfrak m}^4$ and hence the coefficient of $x^2y$ is still a unit.

Using the third column, make $b_{31},b_{32}$ vanish at the expense of introducing $x$ terms in the rows above as in
$$
\left[
\begin{array}{rrlcccc}
a_{11}x+b_{11}y & a_{12}x+b_{12}y & a_{13}x & b_{14}y & b_{15}y & b_{16}y\\
a_{21}x+b_{21}y & a_{22}x+b_{22}y & a_{23}x & b_{24}y & b_{25}y & b_{26}y\\
0 & 0 & y & b_{34}y & b_{35}y & b_{36}y
\end{array}
\right].
$$
This does not change the  $[1~2~3]$ minor. So the coefficient of $x^2y$ in this minor is still a unit. So
$$
\left[
\begin{array}{cc}
a_{11} & a_{12}\\
a_{21} & a_{22}
\end{array}
\right]
$$
is a unit. Multiply by the inverse of this matrix and assume that $M$ is now generated modulo $(I:{\mathfrak m}^2)F$ by
$$
\left[
\begin{array}{rrlcccc}
x+b_{11}y & b_{12}y & a_{13}x & b_{14}y & b_{15}y & b_{16}y\\
b_{21}y & x+b_{22}y & a_{23}x & b_{24}y & b_{25}y & b_{26}y\\
0 & 0 & y & b_{34}y & b_{35}y & b_{36}y
\end{array}
\right].
$$
Now, use the first two columns to convert the $a_{13}x$ and $a_{23}x$ terms to multiples of $y$ and then use the third row to make those vanish. So now, $M$ is now generated modulo $(I:{\mathfrak m}^2)F$ by
$$
\left[
\begin{array}{rrlcccc}
x+b_{11}y & b_{12}y & 0 & b_{14}y & b_{15}y & b_{16}y\\
b_{21}y & x+b_{22}y & 0 & b_{24}y & b_{25}y & b_{26}y\\
0 & 0 & y & b_{34}y & b_{35}y & b_{36}y
\end{array}
\right].
$$
We may now make $b_{34},b_{35},b_{36}$ vanish and assume that $M$ is generated modulo $(I:{\mathfrak m}^2)F$ by
$$
\left[
\begin{array}{rrlcccc}
x+b_{11}y & b_{12}y & 0 & b_{14}y & b_{15}y & b_{16}y\\
b_{21}y & x+b_{22}y & 0 & b_{24}y & b_{25}y & b_{26}y\\
0 & 0 & y & 0 & 0 & 0
\end{array}
\right].
$$
Finally this means that $((I:{\mathfrak m}^2),y)$ is a direct summand of $M$.
\end{proof}

Putting together all our previous results gives us the desired characterisation of ideals that occur as $I(M)$ for an indecomposable integrally closed module $M$ of rank 3.

\begin{theorem} \label{nonexistence} Let $I$ be an integrally closed ${\mathfrak m}$-primary ideal of $R$ with $ord(I) =n$. Then $I = I(M)$ for an indecomposable integrally closed module $M$ of rank $3$ iff one of the following two conditions is satisfied.
\begin{enumerate}
\item $n > 3$.
\item $n=3$ and there do not exist ${\mathfrak m}$-primary integrally closed ideals $J$ and $K$ of orders $1$ and $2$ respectively such that
$I= JK$ and both the conditions below hold:
\begin{enumerate}
\item $\lambda(\frac{R}{JK}) = \lambda(\frac{R}{J}) + \lambda(\frac{R}{K}) + 2$, and
\item ${\mathfrak m}^{2} = {\mathfrak m}J + K$,
\end{enumerate}
\end{enumerate}
\end{theorem}

\begin{proof} If either of the conditions (1) or (2) hold, then by Theorem \ref{existence}, $I = I(M)$ for an indecomposable integrally closed module $M$ of rank 3 (where, in fact, $M$ can be taken to be $M_3(I)$). On the other hand if neither (1) nor (2) holds, then either $n \leq 2$ in which case $I$ is certainly not $I(M)$ for an indecomposable integrally closed module $M$ of rank 3, or $n=3$ and $I=JK$ with $ord(J)=1$, $ord(K) = 2$ and both conditions (2)(a) and (2)(b) holding, in which case, by Proposition \ref{nonexis2}, $I$ is not $I(M)$ for an indecomposable integrally closed module $M$ of rank 3.
\end{proof}

We conclude with a couple of examples illustrating our main results.

\begin{example}
{\rm 
Let $I=(x^m, y)(x,y^p)(x,y^q)$ where $m, p, q \geq 1$. Then the ideals $J=(x^m,y)$ and $K=(x,y^p)(x,y^q)$ satisfy $I=JK$ and 
the two conditions (2)(a) and (2)(b) in Theorem \ref{nonexistence}. Hence, $I$ is not $I(M)$ for an indecomposable integrally closed module $M$ of rank $3$. 
}
\end{example}

\begin{example}
{\rm 
Let $I=(x^m,y)\overline{(x^{p+1},y^{p})}$ where $m, p \geq 2$. If $p \geq 3$, then $ord(I) >3$ and so $I=I(M)$ for an indecomposable integrally closed module $M$ of rank $3$. 
Suppose $p=2$. 
If $J$ and $K$ are $\fkm$-primary integrally closed ideals of orders $1$ and $2$ 
respectively such that $I=JK$, then, by Zariski's unique factorisation theorem, necessarily 
$J=(x^m,y)$ and $K=\overline{(x^3,y^2)}$. Hence, $\lambda(R/I)=m+8>m+7=\lambda(R/J)+\lambda(R/K)+2$. 
Therefore, $I=I(M)$ for an indecomposable integrally closed module $M$ of rank $3$.
}
\end{example}

\begin{example}
{\rm 
Let $I=(x^m,y)\overline{(x^2,y^{3})}$ where $m \geq 2$.  
If we take $J=(x^m,y)$ and $K=\overline{(x^2,y^3)}$, then the ideals $J$ and $K$ satisfy the two conditions (2)(a) and (2)(b) in Theorem \ref{nonexistence}. Hence, $I=(x^m,y)\overline{(x^2,y^3)}$ is not $I(M)$ for an indecomposable integrally closed module $M$ of rank $3$. 
}
\end{example}

\section*{Acknowledgements}
The authors would like to thank the referee for a thorough reading and detailed suggestions. 
The first named author was partially supported by JSPS KAKENHI Grant Number JP20K03535.

\end{document}